\documentclass[11pt]{article}

\usepackage{etex}
\reserveinserts{28}
\usepackage[all]{xy}
\usepackage{enumitem}

\usepackage{amsfonts}
\usepackage{amsthm}
\usepackage{graphicx}
\usepackage{mathrsfs}
\usepackage{bm}
\usepackage{cite}
\usepackage{amssymb,amsmath} 
\usepackage{makecell}
\usepackage{pgf}
\usepackage{pgffor}
\usepackage{pgfcalendar}
\usepackage{pgfpages}
\usepackage{shuffle,yfonts}
\usepackage{mathtools}
\usepackage{tikz}
\usetikzlibrary{arrows,shapes,chains,patterns}

 \allowdisplaybreaks

\theoremstyle{plain}
\newtheorem{thm}{Theorem}[section]
\newtheorem{lem}[thm]{Lemma}
\newtheorem{cor}[thm]{Corollary}

\newtheorem{prop}[thm]{Proposition}

\theoremstyle{definition}

\catcode`!=11
\let\!int\int \def\int{\displaystyle\!int}
\let\!lim\lim \def\lim{\displaystyle\!lim}
\let\!sum\sum \def\sum{\displaystyle\!sum}
\let\!sup\sup \def\sup{\displaystyle\!sup}
\let\!inf\inf \def\inf{\displaystyle\!inf}
\let\!cap\cap \def\cap{\displaystyle\!cap}
\let\!max\max \def\max{\displaystyle\!max}
\let\!min\min \def\min{\displaystyle\!min}
\let\!frac\frac \def\frac{\displaystyle\!frac}
\catcode`!=12

\let\oldsection\section
\renewcommand\section{\setcounter{equation}{0}\oldsection}

\allowdisplaybreaks

\DeclareMathOperator*{\sgn}{{sgn}}

\newcommand{\nc}{\newcommand}
\nc{\baru}{{\bar{u}}}
\nc{\tls}{{\tilde{s}}}
\nc{\tlr}{{\tilde{r}}}
\nc{\tlt}{{\tilde{t}}}
\nc{\tlk}{{\tilde{k}}}
\nc{\tla}{{\tilde{a}}}
\nc{\tlb}{{\tilde{b}}}
\nc{\tlc}{{\tilde{c}}}
\nc{\db}{{\mathbb D}}
\nc{\gs}{\sigma}
\nc{\ol}{\overline}
\nc{\sha}{\shuffle}
\nc{\gd}{\delta}
\nc\tpi{{\tilde{\pi}}}
\nc\ora{\overrightarrow}
\nc\om{{\omega}}
\nc\bfs{{\bf s}}
\nc\bfk{{\boldsymbol{\sl k}}}
\nc\bfl{{\boldsymbol{\sl l}}}
\nc\bfn{{\bf n}}
\nc\bfp{{\bf p}}
\nc\bfq{{\bf q}}
\nc\bfz{{\bf z}}
\nc\bfx{{\bf x}}
\nc\bfB{{\bf B}}
\nc\bfTB{{\bf TB}}
\nc\bfeta{{\boldsymbol \eta}}
\nc\eps{{\varepsilon}}
\nc\bfeps{{\boldsymbol \eps}}
\nc\gl{{\lambda}}
\nc\CMZV{\mathsf{CMZV}}
\nc\AMTV{\mathsf{AMTV}}
\nc{\AMMV}{\mathsf {AMMV}}

\nc\upi{{i}}
\nc\ww{\bm {w}}
\nc\bb{\bm b}
\nc\vv{\bm v}
\nc\UU{\mbox{\bfseries U}}
\nc\FF{\mbox{\bfseries \itshape F}}
\nc\h{\mbox{\bfseries \itshape h}}\nc\dd{\mbox{d}}
\nc\g{\mbox{\bfseries \itshape g}}
\nc\xx{\mbox{\bfseries \itshape x}}
\def\R{\mathbb{R}}
\def\N{\mathbb{N}}
\def\Q{\mathbb{Q}}

\def\z{\zeta}

\def\xx{\left(\frac{1-x}{1+x} \right)}

\nc\bfsi{{\boldsymbol \sigma}}

\nc\divg{{\text{div}}}
\begin{document}
\title {\bf Proof of Kaneko--Tsumura Conjecture \\
on Triple $T$-Values}
\author{Sasha Berger, Aarav Chandra, Jasper Jain, Daniel Xu \\
Department of Mathematics, The Bishop's School, La Jolla, CA 92037, USA\\
sasha.berger.23@bishops.com,aarav.chandra.22@bishops.com,\\
jasper.jain.23@bishops.com,daniel.xu.23@bishops.com\\
\ \\
Ce Xu  \\
School of Math. and Statistics, Anhui Normal University, Wuhu 241000, P.R. China\\
cexu2020@ahnu.edu.cn\\
\ \\
Jianqiang Zhao\\
Department of Mathematics, The Bishop's School, La Jolla, CA 92037, USA\\
zhaoj@ihes.fr}

\date{}
\maketitle \noindent{\bf Abstract} Many $\Q$-linear relations exist between multiple zeta values, the most interesting of which are the weighted sum formulas. In this paper, we generalized these to Euler sums and Kaneko and Tsumura's multiple $T$-values, a variant of the multiple zeta values, by considering generating functions of the Euler sums. Through this approach we are able to re-prove a few known formulas, confirm a conjecture of Kaneko and Tsumura on triple $T$-values, and discover many new identities.




\bigskip

\section{Introduction}
In their seminal work \cite{KanekoTs2019}, Kaneko and Tsumura defined the multiple $T$-values (MTVs),
a variation of the multiple zeta values as follows:
for any positive integers $s_1,\dotsc,s_d$ with $s_1>1$, let
\begin{align}
T(s_1,\dotsc,s_d):=&\, \sum_{\substack{m_1>\dotsm >m_d>0\\ m_j\equiv d-j+1 \pmod{2}} } \frac{2^d}{m_1^{s_1}\dotsm m_d^{s_d}} \notag\\
=&\, \sum_{n_1>\dotsm>n_d>0} \frac{(1+(-1)^{n_1+d}) \dotsm (1+(-1)^{n_{d-1}}) (1-(-1)^{n_d}) }{n_1^{s_1}\dotsm n_d^{s_d}}.\label{equ:defnMTV}
\end{align}
They derived many interesting facts concerning these values with analytical methods and posed some conjectures
when the depth $d=3$. In this short note, we shall prove
the following result which was stated in \cite[Conjecture~4.6]{KanekoTs2019}.
\begin{thm}\label{thm:main} \emph{(=Corollary~\ref{cor:KTconj})}
For all $w\ge 4$
\begin{equation*}
    \sum_{a+b+c=w} 2^b(3^{a-1}-1) T(a,b,c)=\frac23  (w-1)(w-2) T(w).
\end{equation*}
\end{thm}

Observe that all MTVs can be decomposed as linear combinations of Euler sums.
Our strategy to prove Theorem~\ref{thm:main} is to derive similar weighted sum formulas for each Euler sum involved by using
their generating functions. In the process, we also discover a few other types of weighted sum formulas for Euler sums.
We remark that the idea of
proving identities among MZVs using generating functions first appeared in \cite{GanglKaZa2006}.
We also notice that Machide already worked on the MZV case \cite{Machide2013,Machide2015b} using multiple polylogarithms when depth $d\le 4$. Yuan and the last author also studied these directly by using only finite double shuffle relations in \cite{YuanZh2013b} when $d=3$.

\section{Multiple zeta values and Euler sums}
The classical \emph{multiple zeta values} (abbr. MZVs) are defined by (see \cite{Hoffman1992,Zagier1994})
\begin{align*}
\z(s_1,\dotsc,s_d):=\sum\limits_{m_1>\dotsb>m_d>0 } \frac{1}{m_1^{s_1}\dotsm m_d^{s_d}},
\end{align*}
for positive integers $s_1,\dotsc,s_d$ with $s_1>1$. We call $s_1+\dotsb+s_d$ and $d$ the \emph{weight} and \emph{depth}, respectively. A composition $(s_1,\dotsc,s_d)$ is called \emph{admissible} if $s_1>1$.

Euler sums, also known as alternating MZVs, are defined as follows: for $\bfs=(s_1,\dotsc,s_d)\in\N^d$, and $\bfz=(z_1,\dotsc,z_d)\in\{\pm 1\}^d$,
\begin{equation}\label{equ:defnMPL}
\z(s_1,\dotsc,s_d;z_1,\dotsc,z_d):=\sum_{n_1>\cdots>n_d>0}
\frac{z_1^{n_1}\dots z_r^{n_r}}{n_1^{k_1} \dots n_d^{k_d}}
\end{equation}
which converges if and only if $(s_1,z_1)\ne (1,1)$, in which case we call $(\bfs;\bfz)$ \emph{admissible}. It is well-known (see, for e.g., \cite[p. 17]{Zhao2016}) that it has an iterated integral expression
\begin{equation}\label{equ:MPLInt}
\z(s_1,\dotsc,s_d;z_1,\dots,z_d)
=\int_0^1 \left(\frac{d t}{t}\right)^{s_1-1}\frac{d t}{a_1-t} \cdots
\left(\frac{d t}{t}\right)^{s_d-1}\frac{d t}{a_d-t}.
\end{equation}
where $a_k=\prod_{j=1}^k z_j$ for all $k\le d$.
To save space, we put a bar on top of $s_j$ if $z_j=-1$. For example,
\begin{equation*}
\z(\bar2,3,\bar1,4)=\z(2,3,1,4;-1,1,-1,1).
\end{equation*}
It is clear that MZVs are just the special case of Euler sums when all $z_j=1$; namely, no bars can appear.

One of the most important features of Euler sums is that they satisfy many standard relations
(see, for example, \cite[\S 13.3]{Zhao2016}). Among them, the (regularized) double shuffle relations, which are consequences
of the two types of expressions in \eqref{equ:defnMPL} and \eqref{equ:MPLInt},
play a key role in the study of the $\Q$-linear relations among these values.

\section{Regularizations of Euler sums and their generating functions}
Let $M$ be a large positive integer and $\eps>0$ be a very small number. For any composition $\bfs=(s_1,\dots,s_d)$
and $z_1,\dotsc,z_d=\pm 1$, we consider two variations of the Euler sums defined by \eqref{equ:defnMPL} and \eqref{equ:MPLInt}, respectively. First, set
\begin{equation}\label{equ:MPLMForm}
\z^{(M)}(\bfs;\bfz):=\sum_{M\ge n_1>\cdots>n_d>0}
\frac{z_1^{n_1}\dots z_r^{n_r}}{n_1^{s_1} \dots n_d^{s_d}}
\end{equation}
and
\begin{equation}\label{equ:IepsForm}
I^{(\eps)}(\bfs;\bfz):=\int_0^{1-\eps} \left(\frac{d t}{t}\right)^{s_1-1}\frac{d t}{a_1-t} \cdots
\left(\frac{d t}{t}\right)^{s_d-1}\frac{d t}{a_d-t}
\end{equation}
where $a_i=\prod_{j=1}^i z_j^{-1}$. Then $\z^{(M)}$ satisfies the stuffle product
\begin{equation} \label{equ:MPLstuffleDdep2}
\z^{(M)}(s_1;z_1)\z^{(M)}(s_2;z_2)=\z^{(M)}(s_1+s_2;z_1z_2) +
\z^{(M)}(s_1,s_2;z_1,z_2) + \z^{(M)}(s_2,s_1;z_2,z_1).
\end{equation}
On the other hand, by shuffle product of the iterated integrals we see that
\begin{align}
 \label{equ:2Riem1RiemProddep2}
I^{(\eps)}(s_1;z_1)I^{(\eps)}(s_2;z_2)=
\sum_{\substack{t_1,t_2\ge 1\\ t_1+t_2=s_1+s_2}}
    \Bigg[ \binom{t_1-1}{s_2-1} I^{(\eps)}_{t_1,t_2}\Big(z_2,\frac{z_1}{z_2}\Big)
        +  \binom{t_1-1}{s_1-1} I^{(\eps)}_{t_1,t_2}\Big(z_1,\frac{z_2}{z_1}\Big)\Bigg].
\end{align}

It is well-known that for admissible $(\bfs;\bfz)$ we have
\begin{equation*}
\z(\bfs;\bfz)=\lim_{M\to\infty}\z^{(M)}(\bfs;\bfz)=\lim_{\eps\to 0^+} I^{(\eps)}(\bfs;\bfz).
\end{equation*}
Therefore one can derive the so-called double shuffle relations using the two different product structures; namely the stuffle and shuffle products. Then by the usual regularization process one can discover the extremely useful regularized double shuffle relation. Briefly speaking, for every admissible and non-admissible $(\bfs;\bfz)$ of level $N$ there are two polynomials of $T$, denoted by $\z_*(\bfs;\bfz)$ ($*$-regularzed) and $\z_\sha(\bfs;\bfz)$ ($\sha$-regularzed), such that

\begin{enumerate}[label=(\bf{DBSF\arabic*}),leftmargin=2cm]
  \item $\z_*(\bfs;\bfz)\z_*(\bfs';\bfz')$ can be expressed as a $\Q$-linear combination of $*$-regularized colored MZVs of weight $|\bfs|+|\bfs'|$ and level $N$ using the stuffle product.
  \item  $\z_\sha(\bfs;\bfz)\z_\sha(\bfs';\bfz')$ can be expressed as a $\Q$-linear combination of $\sha$-regularized colored MZVs of weight $|\bfs|+|\bfs'|$ and level $N$ using the shuffle product.
  \item \label{page:DBSF3} There is an explicitly defined $\R$-linear map $\rho$ such that $\rho\circ\z_*=\z_\sha$ satisfying $\rho(T)=T$ and $\rho(T^2)=T^2+\z(2).$
\end{enumerate}
We will not go into the details of this theory, instead, we would like to refer the interested reader to \S 13.3.1 of the book \cite{Zhao2016}.

For any fixed alternating signs $\bfz=(z_1,\dotsc,z_d)$ we set $\sgn(\bfz)=(\sgn(z_1),\dotsc,\sgn(z_d))$
and define the generating functions
\begin{equation}\label{equ:defnF}
F_\sharp^{\sgn(\bfz)}(\bfx):=
\sum_{s_1,\dotsc,s_d\ge 1} \z_\sharp(\bfs;\bfz) x_1^{s_1-1}\dotsm x_d^{s_d-1},
\end{equation}
where $\sharp=*$ or $\sha$, $\bfx=(x_1,\dotsc,x_d)$ and $\z_\sharp$ denotes the $\sharp$-regularized value.

\section{Depth 2 weighted sum formulas}
By multiplying $x^{s_1-1}y^{s_2-1}$ on \eqref{equ:MPLstuffleDdep2} and \eqref{equ:2Riem1RiemProddep2}, taking the sum for all $s_1,s_2\in\N$, and specializing at $(z_1,z_2)=(1,1), (1,-1)$ and $(-1,-1)$, respectively, we get after applying the regularization process:
\begin{align*}
\z_\sharp(a)\z_\sharp(b):&\, F_\sha^{+,+}(x+y,y)+ F_\sha^{+,+}(x+y,x)=F_*^{+,+}(x,y)+F_*^{+,+}(y,x)+\frac{F_*^{+}(x)-F_*^{+}(y) }{x-y}, \\
\z_\sharp(a)\z(\bar{b}):&\,  F_\sha^{+,-}(x+y,y)+ F_\sha^{-,-}(x+y,x)=F_*^{+,-}(x,y)+F_*^{-,+}(y,x)+\frac{F_*^{-}(x)-F_*^{-}(y) }{x-y}, \\
\z(\bar{a})\z(\bar{b}):&\,   F_\sha^{-,+}(x+y,y)+ F_\sha^{-,+}(x+y,x)=F_*^{-,-}(x,y)+F_*^{-,-}(y,x)+\frac{F_*^{+}(x)-F_*^{+}(y) }{x-y}.
\end{align*}
Here we need to remark that $\z_\sha(s)=\z_*(s)$ for all $s\in\N$.
Now, replacing $(x,y)$ by $(xt,yt)$ and then comparing the coefficient for $t^{w-2}$ ($w\ge 3$) we immediately
derive the following results. To save space, we set
\begin{equation*}
 \sum =\sum_{a+b=w,a,b\in\N} \quad \text{or} \quad \sum =\sum_{a+b+c=w,a,b,c\in\N}.
\end{equation*}
On the other hand, $\sum{}'$ means that we remove all those terms with $a=1$, i.e.,
\begin{equation*}
 \sum{}' =\sum_{a+b=w,a\ge 2} \quad \text{or} \quad \sum{}' =\sum_{a+b+c=w,a\ge 2}.
\end{equation*}

\begin{prop} \label{prpp:FEdep2}
For any fixed $w\ge 3$, set $f_w(x,y)=\sum_{j=0}^{w-2} x^j y^{w-2-j}$. Then we have
\begin{align*}
\sum \z_\sha(a,b)(x+y)^{\tla}(y^{\tlb}+x^{\tlb})
 =&\,\sum\z_*(a,b)(x^{\tla}y^{\tlb}+y^{\tla}x^{\tlb})+\z(w)f_w(x,y), \\
\sum (x+y)^{\tla}\big(\z_\sha(a,\bar{b})y^{\tlb}+\z(\bar{a},\bar{b})x^{\tlb}\big)
 =&\,\sum \big(\z_*(a,\bar{b})x^{\tla}y^{\tlb}+\z(\bar{a},b)y^{\tla}x^{\tlb}\big)+\z(\bar w)f_w(x,y),  \\
\sum \z(\bar{a},b)(x+y)^{\tla}(y^{\tlb}+x^{\tlb})
 =&\,\sum\z(\bar{a},\bar{b})(x^{\tla}y^{\tlb}+y^{\tla}x^{\tlb})+\z(w)f_w(x,y).
\end{align*}
Here and in the rest of this paper, we set ${\tla}=a-1$, ${\tlb}=b-1$.
\end{prop}

\begin{thm}\label{thm:sumzab}
Let $w\ge 3$ and $v=w-1$. Then we have
\begin{align*}
\sum{}' \z(a,b)=&\, \z(w), \\
\sum{}' \z(\bar{a},\bar{b})=&\, \z(\bar1,v)-\z(\bar1,\bar{v})+\z(\bar w), \\
\sum{}' \z(\bar{a},b) =&\, \z(\bar{v},\bar1)+\z(\bar1,\bar{v})-\z(\bar{v},1)-\z(\bar1,v)+\z(w),\\
\sum{}' \z(a,\bar{b})=&\, \z(\bar{v},1)- \z(\bar{v},\bar1)+\z(\bar w).
\end{align*}
\end{thm}
\begin{proof}
We can prove these by taking $x=1, y=0$ and $x=0, y=1$ in Proposition \ref{prpp:FEdep2}.
\end{proof}

This immediately implies the following corollary about double $T$-values.
\begin{cor} \label{cor:MMVSumFormulaDepth2}
Let $w\ge 3$ and $v=w-1$.  Then we have
\begin{align}\label{equ:DTVsumFormula}
\sum_{a+b=w,a\ge 2} T(a,b)=&\, 2\Big(\z(w)-\z(\bar w)+\z(\bar{v},\bar1)+\z(\bar1,\bar{v})-\z(\bar{v}, 1)-\z(\bar1,v) \Big).
\end{align}
\end{cor}
\begin{proof} We have
\begin{align*}
\sum{}' T(a,b)=
\sum{}'\Big (\z(a,b)+\z(\bar{a},b)-\z(a,\bar{b}) -\z(\bar{a},\bar{b})\Big)
\end{align*}
which reduces to \eqref{equ:DTVsumFormula} quickly by Theorem~\ref{thm:sumzab}.
\end{proof}

With this approach we can now easily recover \cite[Thm. 3.2]{KanekoTs2019}. Notice that
in the proof, we find the depth 2 weighted sum formulas for Euler sums which are analogs of the MZV formula
first discovered by Ohno and Zulidin \cite{OhnoZu2008} and later generalized to
arbitrary depth by Guo and Xie \cite{GuoXi2009}.
\begin{thm}\label{thm:MVT2prWtedSum}
For all $w\ge 3$ we have
\begin{align*}
\sum_{a+b=w,a\ge 2} 2^{a-1} T(a,b)=&\, (w-1)(\z(w)-\z(\bar w))= (w-1) T(w).
\end{align*}
\end{thm}
\begin{proof}
Set $\sum{}'=\sum{}_{a+b=w,a\ge 2}$. Taking $x=1, y=1$ in Proposition \ref{prpp:FEdep2} we get
\begin{align}\label{equ:MVT2prWtedSum}
\sum{}' 2^{\tla}\z(a,b)=&\, \sum{}' \z(a,b)+\frac{w-1}{2}\z(w)=\frac{w+1}{2}\z(w),  \\
\sum{}' 2^{\tla} \big(\z(a,\bar{b})+\z(\bar{a},\bar{b})\big)=&\,
\sum{}'\big(\z(a,\bar{b})+\z(\bar{a},b)\big)+(w-1)\z(\bar w) 
=\z(w)+w\z(\bar w), \notag \\
\sum{}' 2^{\tla}\z(\bar{a},b) =&\,\sum{}' \z(\bar{a},\bar{b})+\frac{w-1}{2}\z(w)
=\z(\bar w)+\frac{w-1}{2}\z(w),  \notag
\end{align}
using the sum formulas in Theorem \ref{thm:sumzab}.
The theorem follows immediately.
\end{proof}

\section{Depth 3 weighted sum formulas, Part A}
As in the depth 2 case, we may derive many identities by the generating functions of the triple Euler sums.
For depth 3, there are two possible ways to produce functional equations of $F_\sharp^{\sgn(\bfz)}(\bfx)$.
We may consider either

\begin{enumerate}[label={(\bf\Alph*)}]
  \item products of double logarithms with single logarithms, denoted by $\z^\sharp_{s_1,s_2}(z_1,z_2) \z^\sharp_{s_3}(z_3)$, or
  \item products of three logarithms, $\z^\sharp_{s_1}(z_1)\z^\sharp_{s_2}(z_2) \z^\sharp_{s_3}(z_3)$.
\end{enumerate}

We start by dealing with case (A) in this section. Observe that
\begin{multline} \label{equ:MPLstuffleDdep3}
\z^{(M)}(s_1,s_2;z_1,z_2)\z^{(M)}(s_3;z_3)=\z^{(M)}(s_1+s_3,s_2;z_1z_3,z_2) + \z^{(M)}(s_1,s_2+s_3;z_1,z_2z_3)\\
\z^{(M)}(s_1,s_2,s_3;z_1,z_2,z_3) + \z^{(M)}(s_1,s_3,s_2;z_1,z_3,z_2) + \z^{(M)}(s_3,s_1,s_2;z_3,z_1,z_2)
\end{multline}
On the other hand, by shuffle product of iterated integrals we see that
\begin{align}
 \label{equ:2Riem1RiemProddep3}
I^{(\eps)}(s_1,s_2;z_1,z_2)I^{(\eps)}(s_3;z_3)=&\,
\sum_{\substack{t_1\ge 2,t_2\ge 1\\ t_1+t_2=s_1+s_3}} \binom{t_1-1}{s_3-1}I^{(\eps)}\Big(t_1,t_2,s_2;z_3,\frac{z_1}{z_3},z_2\Big)\\
+&\, \sum_{\substack{t_1\ge 2,t_2,t_3\ge 1\\ t_1+t_2+t_3 \\  =s_1+s_2+s_3 }}
\binom{t_1-1}{s_1-1} \binom{t_2-1}{s_2-t_3}  I^{(\eps)}\Big(t_1,t_2,t_3;z_1,\frac{z_3}{z_1},\frac{z_1z_2}{z_3}\Big)\\
+&\, \sum_{\substack{t_1\ge 2,t_2,t_3\ge 1\\ t_1+t_2+t_3 \\  =s_1+s_2+s_3 }}
\binom{t_1-1}{s_1-1} \binom{t_2-1}{s_2-1} I^{(\eps)}\Big(t_1,t_2,t_3;z_1,z_2,\frac{z_3}{z_2}\Big).
\end{align}
To save space, we set $\z_\sharp\Big({\bfs \atop \bfz}\Big):=\z_\sharp(\bfs;\bfz)$.
By the usual regularization process \eqref{equ:MPLstuffleDdep3} and \eqref{equ:2Riem1RiemProddep3}
easily lead to the following functional equations for any weight $w\ge 3$ in view of Lemma \ref{lem:relSharp}.
\begin{align}
& \frac{\delta}2\z(2)\z_*\Big({{w-2}\atop{z_3}}\Big)z^{w-3}+\sum_{b+c=w} \Bigg[\left(\frac{x^{\tlb}-z^{\tlb}}{x-z} \right)y^{\tlc}
\z_*\Big({{\ \, b\, \ ,\, c\, }\atop{z_1z_3,z_2}}\Big)
+ x^{\tlb}\left(\frac{y^{\tlc}-z^{\tlc}}{y-z} \right) \z_*\Big({{ b\, ,\ \, c\, \ }\atop{z_1,z_2z_3}}\Big)\Bigg] \notag\\
& +\sum_{a+b+c=w}\Bigg[
 x^{\tla}y^{\tlb}z^{\tlc} \z_*\Big({{\, a\, ,\, b\, ,\, c\, }\atop{z_1,z_2,z_3}}\Big)
+x^{\tla}z^{\tlb}y^{\tlc} \z_*\Big({{\, a\, ,\, b\, ,\, c\, }\atop{z_1,z_3,z_2}}\Big)
+z^{\tla}x^{\tlb}y^{\tlc} \z_*\Big({{\, a\, ,\, b\, ,\, c\, }\atop{z_3,z_1,z_2}}\Big)\Bigg]\notag\\
& =\sum_{a+b+c=w} \Bigg[(x+z)^{\tla} (y+z)^{\tlb}y^{\tlc} \z_\sha\Big({{\, a\, ,\ \, b\, \ ,\ \, c\, \ }\atop{z_1,z_3/z_1,z_1z_2/z_3}}\Big)\notag\\
& \hskip2cm+  (x+z)^{\tla} (y+z)^{\tlb}z^{\tlc} \z_\sha\Big({{\, a\, ,\, b\, ,\ \, c\, \ }\atop{z_1,z_2,z_3/(z_1z_2)}}\Big)
+ (x+z)^{\tla} x^{\tlb}y^{\tlc} \z_\sha\Big({{\, a\, ,\ \, b\, \ ,\, c\, }\atop{z_3,z_1/z_3,z_2}}\Big)\Bigg],
\label{equ:FExyzDtimesSingle}
\end{align}
where $\delta=1$ if $z_1=z_2=1$, and  $\delta=0$ otherwise.
This yields eight cases by different combinations of $z_1,z_2,z_3=\pm1$.

The following lemma will help us to simplify the above formula when regularized values appear.
\begin{lem}\label{lem:relSharp}
Suppose $a,b,c\in\db$ with $|a|+|b|+|c|=w\ge 4$. Then for all $(a,b)\ne (1,1)$ we have
\begin{equation*}
    \z_\sha(a,b,c)=\z_*(a,b,c).
\end{equation*}
Further, for $c=w-2$ or $c=\ol{w-2}$
\begin{equation*}
    \z_\sha(1,1,c)=\z_*(1,1,c)+\frac12 \z(2)\z(c).
\end{equation*}
\end{lem}
\begin{proof}
We know the regularized values $\z_\sha(a,b,c)$ and $\z_*(a,b,c)$ are either constants or
linear polynomials of $T$ if $(a,b)\ne (1,1)$ since the weight is at least 4.
Hence $\z_\sha(a,b,c)=\z_*(a,b,c)$ as $\rho(T)=T$ by (DBSF3) on page \pageref{page:DBSF3}. On the other hand, for all $c\in\db\setminus\{1\}$,
\begin{equation*}
\z_\sha(1,1,c)=\rho(\z_*(1,1,c))=\rho(\z_*(1,1)\z(c)-f(T))
\end{equation*}
for some linear polynomial $f(T)$. Thus, by (DBSF3)
\begin{equation*}
\z_\sha(1,1,c)=\rho\left(\frac{T^2-\z(2)}{2}\right)\z(c)-f(T)=\frac{T^2}{2}\z(c)-f(T)=\z_*(1,1,c)+\frac12 \z(2)\z(c)
\end{equation*}
as desired.
\end{proof}

\begin{thm}\label{thm:sumzab1}
Let $w\ge 4$, $u=w-2$ and $v=w-1$. Then we have
\begin{align*}
\sum{}' \z(a,b,1)=&\,\z(v,1)+\z(u,2),\\
\sum{}' \z(a,\bar{b},1)=&\,\z(\ol{v},1)+\z(\bar{u},2)+2\z(\bar{u},1,1)-\z(\bar{u},\bar1,\bar1)-\z(\bar{u},1,\bar1),\\
\sum{}' \z(\bar{a},b,1)=&\,\z(v,1)+\z(\bar{u},\bar2)+\z(\bar{u},1,\bar1)+\z(\bar{u},\bar1,1)+\z(\bar1,\bar{u},1)-2\z(\bar{u},1,1)-\z(\bar1,u,1),\\
\sum{}' \z(\bar{a},\bar{b},1)=&\,\z(\ol{v},1)+\z(u,\bar2)+\z(u,\bar1,1)+\z(\bar1,u,1)-\z(u,\bar1,\bar1)-\z(\bar1,\bar u,1),\\
\sum{}' \z(\bar{a},\bar{b},\bar1)=&\,\z(v,\bar1)+\z(\bar u,2)+\z(\bar u,\bar1,\bar1)-\z(\bar u,1,\bar1),\\
\sum{}' \z(a,b,\bar1)=&\,\z(v,\bar1)+\z(u,\bar2)+\z(u,\bar1,1)-\z(u,\bar1,\bar1),\\
\sum{}' \z(a,\bar{b},\bar1)=&\,\z(\ol{v},\bar1)+\z(\bar{u},\bar2)+\z(\bar{u},1,\bar1)-\z(\bar{u},\bar1,1),\\
\sum{}' \z(\bar{a},b,\bar1)=&\,\z(v,\bar1)+\z(\bar{u},2)+\z(\bar{u},\bar1,\bar1)+\z(\bar1,\bar{u},\bar1)-\z(\bar{u},1,\bar1)-\z(\bar1,u,\bar1).
\end{align*}
\end{thm}
\begin{proof}
Taking $x=1,y=z=0$ in \eqref{equ:FExyzDtimesSingle} we obtain the formulas in the theorem immediately using Lemma~\ref{lem:relSharp}.
\end{proof}

\begin{cor}\label{cor:MMVab1SumFormula}
Let $w\ge 4$, $u=w-2$ and $v=w-1$.  Then we have
\begin{align*}
\sum{}' T(a,b,1)=&\, 2T(u,2)+4\Big(\z(u,\bar1,\bar1)+\z(\bar{u},1,1)-\z(\bar{u},1,\bar1)-\z(u,\bar1,1)\Big).
\end{align*}
\end{cor}
\begin{proof}
These follow from direct computation using Theorem~\ref{thm:sumzab1} and the two identities:
\begin{equation}\label{equ:z11rel}
\z(\bar{1},\bar1)-\z(\bar{1},1)=\z(\bar2) \quad \text{and} \quad 2\z(\bar2)=-\z(2).
\end{equation}
See, for example, \cite[Prop.~14.2.5]{Zhao2016}.
\end{proof}

\begin{thm}\label{thm:sumz1bc}
Let $w\ge 4$, $u=w-2$ and $v=w-1$. Set $\sum{}=\sum{}_{b+c=w}$. Then we have
\begin{align*}
\sum{}\z_\sha(1,b,c)=&\,\z(2,u)+\z_*(1,v)+\z_*(1,1,u)-\frac12\z(2)\z(u),\\
\sum{}\z_\sha(1,\bar{b},c)=&\, \z(\bar2,\bar{u})
+\z_*(1,v)+\z_*(1,\bar1,\bar{u})+\z_*(1,\bar{u},\bar1)+\z(\bar1,1,\bar{u})-\z(\bar1,\bar1,\bar{u})-\z_*(1,\bar{u},1),\\
\sum{}\z(\bar1,b,c)=&\, \z(2,u)+\z(\bar1,\bar{v})
+\z(\bar1,u,\bar1)+2\z(\bar1,\bar1,u)-\z(\bar1,1,u)-\z(\bar1,u,1),\\
\sum{}\z(\bar1,\bar{b},c)=&\, \z(\bar2,\bar{u})+\z(\bar1,\bar{v})+\z(\bar1,1,\bar{u}),\\
\sum{}\z(\bar1,\bar{b},\bar{c})=&\, \z(\bar2,u)
+\z(\bar1,v)+\z(\bar1,u,1) +\z(\bar1,1,u)-\z(\bar1,u,\bar1),\\
\sum{}\z_\sha(1,b,\bar{c})=&\, \z(2,\bar{u})
+\z_*(1,\bar{v})+\z_*(1,\bar{u},1)+\z_*(1,1,\bar{u})
-\z_\sha(1,\bar{u},\bar1)-\frac12\z(2)\z(\bar{u}),\\
\sum{}\z_\sha(1,\bar{b},\bar{c})=&\, \z(\bar2,u)
+\z_*(1,\bar{v})+\z_*(1,\bar1,u)+\z(\bar1,1,u)-\z(\bar1,\bar1,u),\\
\sum{}\z(\bar1,b,\bar{c})=&\, \z(2,\bar{u})
+\z(\bar1,v)+2\z(\bar1,\bar1,\bar{u})-\z(\bar1,1,\bar{u}).
\end{align*}
\end{thm}
\begin{proof}
Taking $y=1,x=z=0$ in \eqref{equ:FExyzDtimesSingle} we obtain the formulas in the theorem immediately by Lemma~\ref{lem:relSharp}.
\end{proof}

Utilizing the same ideas as above, we may now derive the sum formulas for all triple Euler sums
and therefore the sum formula for triple $T$-values.
\begin{thm}\label{thm:sumzabc}
Let $w\ge 4$ and $u=w-2$. Then we have
\begin{align*}
\sum{}' T(a,b,c)=&\, \frac23 T(2)T(u)-2T(u,2)+4\Big(\z(u,\bar1,1)-\z(\baru,1,1)+\z(\baru,1,\bar1)-\z(u,\bar1,\bar1)\Big).
\end{align*}
\end{thm}

\begin{proof}
Taking $x=y=0,z=1$ in \eqref{equ:FExyzDtimesSingle} and using Lemma~\ref{lem:relSharp}
we can find the sum formulas for all triple Euler sums as follows:
\begin{align*}
\sum{}'\z(a,b,c)=&\,\z(w),\\
\sum{}'\z(a,\bar{b},c)=&\,2\z(\bar{u},1,\bar1)-2\z(\bar{u},1,1)-\z(\bar1,1,\bar{u})+\z(\bar1,\bar1,\bar{u})\\
 \ & +\z(\bar{v},\bar1)-\z(\bar{v},1)-\z(\bar{u},2)-\z(\bar2,\bar{u}),\\
\sum{}'\z(\bar{a},b,c)=&\,\z(\bar{u},1,1)-\z(\bar{u},1,\bar1)+\z(\bar1,u,1)-\z(\bar1,u,\bar1)\\
 \ & +\z(\bar1,1,\bar{u})+\z(\bar1,1,u)-2\z(\bar1,\bar1,u)-\z(\bar{u},\bar2)-\z(2,u)\\
\sum{}'\z(\bar{a},\bar{b},c)=&\,\z(\bar1,1,u)-\z(\bar1,1,\bar{u})+\z(\bar1,v)-\z(\bar1,\bar{v})
+\z(\bar2,u)-\z(\bar2,\bar{u})+\z(\bar{w}),\\
\sum{}' \z(\bar{a},b,\bar{c})=&\,\z(u,\bar1,1)-\z(u,\bar1,\bar1)+\z(\bar1,1,\bar{u})-2\z(\bar1,\bar1,\bar{u})+\z(\bar1,\bar1,u) \\
 \ &+ \z(\bar1,\bar{v})-\z(\bar1,v)-\z(u,2)-\z(2,\bar{u})-\z(u)\z(\bar2),\\
\sum{}'\z(a,\bar{b},\bar{c})=&\,\z(u,\bar1,\bar1)-\z(u,\bar1,1)+\z(\bar1,\bar1,u)-\z(\bar1,1,u)-\z(u,\bar2)-\z(\bar2,u),\\
\sum{}'\z(a,b,\bar{c})=&\,\z(\bar{u},1,1)-\z(\bar{u},1,\bar1)+\z(\bar{v},1)-\z(\bar{v},\bar1)-\z(\bar{u},\bar2)+\z(\bar{u},2)+\z(\bar{w}),\\
\sum{}' \z(\bar{a},\bar{b},\bar{c})=&\,\z(\bar1,u,\bar1)-\z(\bar1,u,1)+\z(\bar1,\bar1,\bar{u})-\z(\bar1,1,u)-\z(\bar{u},2)-\z(\bar2,u).
\end{align*}
In the process we need Theorems~\ref{thm:sumzab1} and~\ref{thm:sumz1bc}. The theorem follows immediately from the decomposition
of the triple $T$-values in terms of the triple Euler sums.
\end{proof}

 From the above we can derive \cite[Thm. 3.3]{KanekoTs2019} as a corollary.
\begin{cor}\label{cor:KanecoTsumuraThm3.3}
For all $w\ge 4$ we have
\begin{equation*}
    \sum_{a+b+c=w,a\ge 2} T(a,b,c)+ \sum_{a+b=w,a\ge 2} T(a,b,1)=\frac23T(2) T(w-2).
\end{equation*}
\end{cor}
\begin{proof}
Let $u=w-2$. We only need to remove the non-admissible terms from those eight equations in
the proof of Theorem~\ref{thm:sumzabc} by Theorem~\ref{thm:sumz1bc}.
\end{proof}

The following restricted sum formulas for triple Euler sums will be needed in the future.
\begin{thm}\label{thm:sumza1c}
Let $w\ge 4$, $u=w-2$ and $v=w-1$. Set $\sum{}=\sum{}_{a+c=w-1}$. Then we have
\begin{align*}
\sum{}\z_*(a,1,c)&\,=\z_*(1,1,u)+\z(v,1)+\z(2,u),\\
\sum{}\z_*(a,\bar{1},\bar{c})&\,=\z_*(1,\bar1,\bar{u})+\z(\bar1,\bar{u},\bar1)-\z(\bar1,\bar{u},1)+\z(v,\bar1)+\z(\bar2,\bar{u}),\\
\sum{}\z(\bar{a},\bar{1},c)&\,=\z(\bar1,\bar1,\bar{u})+\z(2,\bar{u}) +\z(\bar{v},\bar1),\\
\sum{}\z(\bar{a},1,\bar{c})&\,=\z(\bar1,u,1) +\z(\bar1,1,u)-\z(\bar1,u,\bar1)+\z(\bar{v},1) +\z(\bar2,u), \\
\sum{}\z(\bar{a},1,c)&\,=\z(\bar1,\bar1,\bar{u})+2\z(\bar{v},1)+\z(\bar1,v)-\z(\bar1,\bar{v})-\z(\bar{v},\bar1)-\z(\bar{u},2)-\z(w),\\
\sum{}\z_*(a,1,\bar{c})&\,=\z(\bar{u},\bar1,\bar1)-\z(\bar{u},\bar1,1)+\z_*(1,1,\bar{u})-\z(\bar{v},\bar1)\\
&\, +\z(v,\bar1)+2\z(\bar{v},\bar1)-\z(\bar{v},1)-\z(\bar{u},\bar2)+\z(2)\z(\bar{u})-\z(\bar{w}),\\
\sum{}\z_*(a,\bar{1},c)&\,=2\z(u,\bar1,\bar1)-2\z(u,\bar1,1)+\z(\bar1,1,u)+\z_*(1,\bar1,u)-\z(\bar1,\bar1,u)+\z(\bar{v},\bar1)\\
&\,     +\z(v,\bar1)+\z(\bar1,v)-\z(\bar{v},1)-\z(\bar1,\bar{v})-\z(2,u)-2\z(u,\bar2)-\z(\bar{w})- \z(w),\\
\sum{}\z(\bar{a},\bar{1},\bar{c})&\,=3\z(u,\bar1,1)-3\z(u,\bar1,\bar1)+ \z(\bar1,\bar1,u)
  +2\z(\bar{v},1)-\z(\bar{v},\bar1)+\z(u,\bar2)+\z(2,u)-\z(u,2).
\end{align*}
\end{thm}
\begin{proof}
Taking $x=y=1,z=0$ in \eqref{equ:FExyzDtimesSingle} we obtain the formulas in the theorem using Lemma~\ref{lem:relSharp},
Theorems~\ref{thm:sumzab}, \ref{thm:sumzab1}, \ref{thm:sumz1bc}, and the sum formulas in the proof of Theorem~\ref{thm:sumzabc}.
\end{proof}

The following result will play a key role in the proof of Theorem \ref{thm:main}.
\begin{thm}\label{thm:sumzabcWt2bPOWER}
Let $w\ge 4$, $u=w-2$ and $v=w-1$. Then we have
\begin{align*}
& \sum{} 2\cdot 2^\tlb  \z_\sha(a,b,c)=2\z_*(1,1,u)+w\z_*(1,v)-2\z(u,2)-\z(w),\\
& \sum{} 2\cdot 2^\tlb \z_\sha(a,\bar b,c)=4\z(u,\bar1,\bar1)-4\z(u,\bar1,1)-2\z(\bar1,\bar1,u)+2\z(\bar1,1,u)+2\z_*(1,\bar1,u)\\
&\ \hskip1.5cm	    +2\z(\bar{v},\bar1)-2\z(\bar{v},1)+u\z_*(1,v)+2\z_*(1,\bar{v})+4\z(\bar2,u)+2\z(u,2)+\z(w)+2\z(\bar{w}),\\
& \sum{} 2\cdot 2^\tlb \z(\bar a,b,c)
=   2\z(\bar1,\bar1,\bar{u}) +w\z(\bar1,\bar{v})+2\z(\bar{v},\bar1)-2\z(\bar{v},1)\\
&\ \hskip1.5cm	 +\z(\bar2,\bar{u})+\z(2,\bar{u})+2\z(w)-\z(\bar{w}) +\z(2)\z(\bar{u}),\\
& \sum{} 2\cdot 2^\tlb \z(\bar a,\bar b,c)=2\z(\bar1,\bar1,\bar{u})+u\z(\bar1,\bar{v})+2\z(\bar1,v)+\z(\bar2,\bar{u})+\z(2,\bar{u})+\z(\bar{w}),\\
& \sum{} 2^\tlb \Big(\z_\sha(\bar a,\bar b,\bar c)
+\z_\sha(\bar a,b,\bar c) \Big)
    =2\z(u,\bar1,1)+2\z(\bar1,\bar1,u)-2\z(u,\bar1,\bar1)\\
&\ \hskip1.5cm	+v\z(\bar1,v)+\z(\bar{v},1)+\z(\bar1,\bar{v})-\z(\bar{v},\bar1)+\z(2,u)-\z(u,2),\\
& \sum{} 2^\tlb \Big(\z_\sha(a,b,\bar c)
+\z_\sha(a,\bar b,\bar c) \Big)=\z_*(1,\bar1,\bar{u})+\z(\bar1,1,\bar{u})-\z(\bar1,\bar1,\bar{u})+\z_*(1,1,\bar{u})\\
&\ \hskip1.5cm	 -\z(\bar{v},\bar1)+\z(\bar{v},1)+\z_*(1,v)+v\z_*(1,\bar{v})+\z(\bar2,\bar{u})+\z(2,\bar{u})+\z(\bar{w})+3\z(\bar2)\z(\bar{u}).
\end{align*}
\end{thm}

\begin{proof}
Taking $x=0, y=z=1$ in \eqref{equ:FExyzDtimesSingle} we obtain the formulas in the theorem using Lemma~\ref{lem:relSharp},
Theorems~\ref{thm:sumzab}, \ref{thm:sumzab1}, \ref{thm:sumz1bc}, and \ref{thm:sumza1c}..
\end{proof}

\section{Depth 3 weighted sum formulas, Part B}
Similarly, by considering case (B) $\z_\sharp(s_1)\z_\sharp(s_2)\z_\sharp(s_3)$ ($\sharp=*$ or $\sha$, $s_j\in\db$) we may arrive at
the following, where $\sum_{}=\sum_{a+b+c=w,a,b,c\ge1}$ as before and $\sum{}'=\sum_{k+c=w,k\ge 2}$:
\begin{align}
 &  \sum{}\Big[(x+y+z)^{\tla}(z+y)^{\tlb}y^{\tlc}\z_\sha\Big({{\,a\, ,\ \,b\,\ ,\ \, c\,\  }\atop{z_1,z_3/z_1,z_2/z_3}}\Big)
  + (x+y+z)^{\tla}(z+y)^{\tlb}z^{\tlc}\z_\sha\Big({{\,a\, ,\ \,b\,\ ,\ \, c\,\  }\atop{z_1,z_2/z_1,z_3/z_2}}\Big) \notag \\
 &+ (x+y+z)^{\tla}(z+x)^{\tlb}x^{\tlc}\z_\sha\Big({{\,a\, ,\ \,b\,\ ,\ \, c\,\  }\atop{z_2,z_3/z_2,z_1/z_3}}\Big)
  + (x+y+z)^{\tla}(z+x)^{\tlb}z^{\tlc}\z_\sha\Big({{\,a\, ,\ \,b\,\ ,\ \, c\,\  }\atop{z_2,z_1/z_2,z_3/z_1}}\Big)\notag \\
 &+ (x+y+z)^{\tla}(x+y)^{\tlb}y^{\tlc}\z_\sha\Big({{\,a\, ,\ \,b\,\ ,\ \, c\,\  }\atop{z_3,z_1/z_3,z_2/z_1}}\Big)
  + (x+y+z)^{\tla}(x+y)^{\tlb}x^{\tlc}\z_\sha\Big({{\,a\, ,\ \,b\,\ ,\ \, c\,\  }\atop{z_3,z_2/z_3,z_1/z_2}}\Big) \Big] \notag \\
 & =\sum{}\Big[ x^{\tla}y^{\tlb} z^{\tlc}\z_*\Big({{\,a\, ,\,b\, ,\, c\, }\atop{z_1,z_2,z_3}}\Big)
  +x^{\tla}z^{\tlb} y^{\tlc}\z_*\Big({{\,a\, ,\,b\, ,\, c\, }\atop{z_1,z_3,z_2}}\Big)
  +y^{\tla}x^{\tlb} z^{\tlc}\z_*\Big({{\,a\, ,\,b\, ,\, c\, }\atop{z_2,z_1,z_3}}\Big)\notag \\
 &  +y^{\tla}z^{\tlb} x^{\tlc}\z_*\Big({{\,a\, ,\,b\, ,\, c\, }\atop{z_2,z_3,z_1}}\Big)
 +z^{\tla}x^{\tlb} y^{\tlc}\z_*\Big({{\,a\, ,\,b\, ,\, c\, }\atop{z_3,z_1,z_2}}\Big)
  +z^{\tla}y^{\tlb}x^{\tlc}\z_*\Big({{\,a\, ,\,b\, ,\, c\, }\atop{z_3,z_2,z_1}}\Big) \Big]\notag  \\
 &+ \sum{}'\Bigg[\Big(\z_*\Big({{\ \,k\,\ ,\, c\, }\atop{z_1z_2,z_3}}\Big)+\z_*\Big({{\,c\, ,\ \, k\,\ }\atop{z_3,z_1z_2}}\Big)\Big)\sum_{a+b=k} x^{\tla} y^{\tlb}  z^{\tlc}
 +  \Big(\z_*\Big({{\ \,k\,\ ,\, c\, }\atop{z_3z_1,z_2}}\Big)+\z_*\Big({{\,c\, ,\ \, k\,\  }\atop{z_2,z_3z_1}}\Big)\Big)\sum_{a+b=k} x^{\tla} z^{\tlb}  y^{\tlc} \notag \\
 &+ \Big(\z_*\Big({{\ \,k\,\ ,\, c\, }\atop{z_2z_3,z_1}}\Big)+\z_*\Big({{\,c\, ,\ \, k\,\ }\atop{z_1,z_2z_3}}\Big)\Big)\sum_{a+b=k} y^{\tla} z^{\tlb}  x^{\tlc}\Bigg]+ \z\Big({{\ \ w\ \ }\atop{z_1z_2z_3}}\Big)\sum  x^{\tla}y^{\tlb}z^{\tlc} \label{equ:tripleProd}
\end{align}
which leads to essentially four different cases using various combinations of $z_1,z_2,z_3=\pm1$.

\begin{thm}\label{thm:sumzabcWt3a2bPOWER}
Let $w\ge 4$, $u=w-2$ and $v=w-1$.  Then we have
\begin{align*}
&\ 2 \sum 3^\tla 2^\tlb \z_\sha(a,b,c)=\left(\frac13\binom{v}{2}+u\right)\z(w)+2\z_*(1,1,u)-2\z(u,2)+w\z_*(1,v),\\
&\ 2\sum 3^\tla 2^\tlb \Big( \z_\sha(a,\bar b, c)+\z(\bar a,b,\bar c)+\z(\bar a,\bar b,\bar c) \Big)
=2\z(\bar1, 1,u)+2\z_*(1,\bar1,u)+2\z(\bar1,\bar1,u)+\binom{v}{2}\z(w)\\
&\hskip1cm +2\z(\bar1,v)+2\z_*(1,\bar{v})+2\z(\bar1,\bar{v})-2\z(u,2) -4\z(u,\bar2)
+ u\Big( 2\z(\bar1,v)+2\z(\bar{w})+\z(w)+\z_*(1,v) \Big),\\
&\ 2\sum 3^\tla 2^\tlb \Big( \z_\sha(a,\bar b,\bar c)
+\z_\sha(a,b,\bar c)+\z(\bar a,\bar b,c) \Big) =
2\z_*(1,1,\bar{u})+2\z_*(1,\bar1,\bar{u})+2\z(\bar1,1,\bar{u})\\
&\hskip1cm +2\z_*(1,\bar{v})+2\z(\bar1,\bar{v})+2\z_*(1,v)
+\z(\bar2,\bar{u})-\z(2,\bar{u})-4\z(\bar{u},2)-2\z(\bar{u},\bar2)\\
&\hskip1cm +u\Big(2\z_*(1,\bar{v})+\z(\bar1, \bar{v})+\z(w)+2\z(\bar{w})\Big)+ \binom{v}{2}\z(\bar{w})+\z(w)-\z(\bar{w}),\\
&\ 2\sum 3^\tla 2^\tlb \z(\bar a,b,c)=2\z(\bar1,\bar1,\bar{u})+u\z(\bar1,\bar{v})+2\z(\bar1,v)\\
&\hskip1cm +\z(2,\bar{u})-2\z(\bar{u},\bar2)-\z(\bar2,\bar{u})+\left(\frac13\binom{v}{2}+1\right)\z(\bar{w})+(w-3)\z(w).
\end{align*}
\end{thm}
\begin{proof}
Taking $x=y=z=1$ in \eqref{equ:tripleProd} we get the following four equations:
\begin{align*}
 & 6 \sum 3^\tla 2^\tlb \z_\sha(a,b,c)
=\binom{v}{2}\z(w)+6\sum \z_*(a,b,c)
+3\sum{}\tlk\big[ \z_*(k,c)+\z_*(c,k) \big], \\
 & 2 \sum 3^\tla 2^\tlb \Big( \z_\sha(a,\bar b, c)+\z_\sha(\bar a,b,\bar c)
+\z_\sha(\bar a,\bar b,\bar c) \Big) =
 2 \sum \Big( \z(\bar a,b,\bar c)
+\z(\bar a,\bar b,c)
+\z(\bar a,b,\bar c) \Big) \\
 &\ \quad+\binom{v}{2}\z(w)+\sum{}\tlk\Big( \big[ 2\z(\bar k,\bar c)+2\z(\bar c,\bar k)\big]
+\big[ \z_*(k,c)+\z_*(c,k)\big] \Big) , \\
 & 2\sum 3^\tla 2^\tlb \Big( \z_\sha(a,\bar b,\bar c)
+\z_\sha(a,b,\bar c)+\z_\sha(\bar a,\bar b,c) \Big)
 =2 \sum \Big( \z_*(a,b,\bar c)
+\z_*(a,\bar b,c)
+\z(\bar a,b,c) \Big) \\
 &\ \quad+\binom{v}{2} \z(\bar w)+\sum{}\tlk\Big(2\big[\z_*(k,\bar c)+\z_*(c,\bar k)\big]+\big[\z(\bar c,k)
+\z(\bar k,c)\big]
+\z(\bar k,c)-\z_*(k,\bar c)\Big),\\
 & 6\sum 3^\tla 2^\tlb \z_\sha(\bar a,b,c)
 = \binom{v}{2}\z(\bar w)+6 \sum\z(\bar a,\bar b,\bar c)
+3\sum{}\tlk\Big(\big[ \z(\bar k,c)+\z(\bar c,k) \big]+\z_*(k,\bar c)-\z(\bar k,c)\Big).
\end{align*}
Here we have used the fact that $\sum_{k\ge 2}\tlk(\cdots)=\sum_{k\ge 1}\tlk(\cdots)$.
Exchanging indices $k$ and $c$ for the second term in each of the square brackets and then setting $(k,c)\to (a,b)$,
we find that
\begin{align*}
 & 2 \sum 3^\tla 2^\tlb \z_\sha(a,b,c)
=\binom{v}{2}\frac{\z(w)}3+2\sum \z_*(a,b,c)
+(w-2)\sum{} \z_*(a,b) , \\
 & 2\sum 3^\tla 2^\tlb \Big( \z_\sha(a,\bar b, c)+\z(\bar a,b,\bar c)
+\z(\bar a,\bar b,\bar c) \Big) =
 2\sum \Big( \z(\bar a,b,\bar c)
+\z(\bar a,\bar b,c)
+\z(a,\bar b,\bar c) \Big) \\
 & \ \qquad+\binom{v}{2}\z(w)+(w-2)\sum{} \Big( 2\z(\bar a,\bar b)+\z_*(a,b) \Big) , \\
 & 2\sum 3^\tla 2^\tlb \Big( \z_\sha(a,\bar b,\bar c)
+\z_\sha(a,b,\bar c)+\z(\bar a,\bar b,c) \Big)
 = 2\sum \Big(\z_*(a,b,\bar c)+\z_*(a,\bar b,c)+\z(\bar a,b,c) \Big) \\
 &\ \qquad+\binom{v}{2}\z(\bar w)+(w-2)\sum{} \Big(2\z_*(a,\bar b)+\z(\bar a,b)\Big)
+\sum{}' \tla \Big(\z_*(\bar a,b)-\z(a,\bar b)\Big),\\
 & 2\sum 3^\tla 2^\tlb \z(\bar a,b,c)
 = \binom{v}{2}\frac{\z(\bar w)}3+2 \sum \z(\bar a,\bar b,\bar c)
+(w-2)\sum{} \z_*(\bar a,b)+\sum{}'\tla\Big(\z_*(a,\bar b)-\z(\bar a,b)\Big).
\end{align*}
The theorem follows easily from Theorems \ref{thm:sumzab}, \ref{thm:sumz1bc} and
\ref{thm:sumzabc}.
\end{proof}

Finally, we are able to prove Kaneko--Tsumura's conjecture on triple $T$-values.
\begin{cor} \label{cor:KTconj}
For all $w\ge 4$
\begin{equation*}
    \sum_{a+b+c=w} 2^b(3^{a-1}-1) T(a,b,c)=\frac23  (w-1)(w-2) T(w).
\end{equation*}
\end{cor}
\begin{proof}
Set $u=w-2$ and $v=w-1$ as before. By Theorems~\ref{thm:sumzabcWt2bPOWER} and \ref{thm:sumzabcWt3a2bPOWER} we see that
\begin{align*}
  & \sum_{a+b+c=w} 2^b(3^{a-1}-1) T(a,b,c)=  \frac23(w-1)(w-2)\Big(\z(w)-\z(\bar{w})\Big)  -2\z(w)+2\z(\bar{w})\\
  & \hskip1cm +2\z(2)\z(u)-3\z(2)\z(\bar{u})-2\z(\bar{u},\bar1,\bar1)-2\z(\bar1,\bar1,\bar{u})-2\z(\bar1,\bar{u},\bar1)-2\z(\bar1,v) -2\z(v,\bar1)\\
  & \hskip1cm +2\z(\bar1,1,\bar{u})+2\z(\bar{u},\bar1,1)+2\z(\bar1,\bar{u},1)+6\z(\bar1,u,1)+6\z(u,\bar1,1)+6\z(\bar1,1,u)+6\z(\bar1,v)\\
  & \hskip1cm -6\z(\bar1,\bar1,u)-6\z(\bar1,u,\bar1)-6\z(u,\bar1,\bar1)-6\z(\bar{v},\bar1)-4\z(\bar1,\bar{v})+2\z(v,1)+6\z(\bar{v},1)\\
  & \hskip1cm -2\z(u,\bar2)-2\z(\bar2,u)+4\z(\bar{u},2)+4\z(\bar2,\bar{u})+4\z(\bar{u},\bar2)+4\z(2,\bar{u})-6\z(2,u)-6\z(u,2).
\end{align*}
Each of the last four lines can be simplified further by stuffle relations so that we get
\begin{align*}
  & \sum_{a+b+c=w} 2^b(3^{a-1}-1) T(a,b,c)= \frac43\binom{v}{2}\Big(\z(w)-\z(\bar{w})\Big)
-4\z(2)\z(u)+\z(2)\z(\bar{u})\\
  & \hskip1cm+4\z(\bar2)\z(\bar{u})-2\z(\bar2)\z(u)
-2\z(\bar{u})\big(\z(\bar1,\bar1)-\z(\bar1,1)\big)
+6\z(u)\big(\z(\bar1,1)-\z(\bar1,\bar1)\big).
\end{align*}
So the corollary follows immediately from \eqref{equ:z11rel}.
\end{proof}

{\bf Acknowledgments.}  Ce Xu is supported by the Scientific Research Foundation for Scholars of Anhui Normal University.

\end{document}